\documentclass[reqno]{amsart}
\usepackage{hyperref}

\usepackage{a4wide,calrsfs}
\usepackage{amsmath,bbm}
\usepackage{color}
\usepackage{epsfig}
\usepackage{subfigure}
\usepackage{amssymb}
\usepackage{enumerate}
\usepackage{graphicx}
\usepackage{subfigure}
\usepackage{amsmath}
\usepackage{rotating}
\usepackage{stmaryrd}
\usepackage{upgreek}

\begin{document}
\title[Almost automorphic delayed equations]
{Almost automorphic delayed differential  equations and Lasota-Wazewska model}

\author[A. Coronel, Ch. Maul\'en, M. Pinto, D. Sepulveda]
{An{\'\i}bal Coronel, Christopher Maul\'en, Manuel Pinto, Daniel Sepulveda}  

\address{An\'ibal Coronel \newline
GMA, Departamento de Ciencias B\'asicas, Facultad de Ciencias, Universidad del B\'{\i}o-B\'{\i}o, Campus Fernando May, Chill\'{a}n, Chile.}
\email{acoronel@ubiobio.cl}

\address{Christopher Maul\'en \newline
Departamento de Matem\'aticas, Facultad de Ciencias, Universidad de Chile}
\email{bleick@ug.uchile.cl}

\address{Manuel Pinto \newline
Departamento de Matem\'aticas, Facultad de Ciencias, Universidad de Chile}
\email{pintoj.uchile@gmail.com}

\address{Daniel Sepulveda \newline
Escuela de Matem\'aticas y Estad\'isticas, Universidad Central de Chile}
\email{daniel.sep.oe@gmail.com}

\date{\today}

\thanks{Partially supported by FONDECYT 1120709. 
CONICYT-PCHA/Mag\'ister Nacional/2013-221320155. An{\'\i}bal Coronel 
thanks for the  support of research projects 124109 3/R, 104709 01 F/E and
121909 GI/C at Universidad del B{\'\i}o-B{\'\i}o, Chile.}


\begin{abstract}
Existence of almost automorphic solutions for abstract delayed
differential equations is established. Using ergodicity,
exponential dichotomy and Bi-almost automorphicity 
on the homogeneous part, sufficient
conditions for the existence and uniqueness of almost automorphic
solutions are given.
\end{abstract}
\maketitle

\numberwithin{equation}{section}
\newtheorem{thm}{Theorem}[section]
\newtheorem{lem}[thm]{Lemma}
\newtheorem{prop}[thm]{Proposition}
\newtheorem{cor}[thm]{Corollary}
\newtheorem{defn}{Definition}[section]
\newtheorem{conj}{Conjecture}[section]
\newtheorem{exmp}{Example}[section]
\newtheorem{rem}{Remark}[section]
\allowdisplaybreaks

\section{Introduction}

The development of the theory of almost periodic type functions has been
strongly stimulated
by problems arising in differential equations,  stability theory, 
dynamical systems and many other areas of science.
Nowadays,  there exist also a wide range of applications starting 
from the basic mathematical models based on linear ordinary 
differential equations, including nonlinear linear ordinary 
differential equations, differential equations in Banach space
and also partial differential equations. Moreover, there exist 
several related concepts which arise as generalizations of the
almost periodic concept. For instance the notions of almost automorphic,
asymptotically almost periodic, asymptotically almost automorphic
and pseudo almost periodic. Since there are plenty of 
results in literature, let us just quote, for their applications in
engineering and life science, for example 
asymptotically almost periodic functions 
\cite{Hernandez&Santos,Henriquez&Pierri&Taboas_1,
Henriquez&Pierri&Taboas_2,liang,Nicola&Pierre,
pinto_robledo_nonl,pinto_robledo_jmaa,pinto_robledo_applmat,pinto_robledo_glass,
Utz&Waltman,yoshizawa},
and pseudo almost periodic functions
\cite{Cuevas&Pinto,diagana,pinto_3}. 
Moreover, we recall that N'Gu\'er\'ekata 
has given a huge impulse to the study of almost automorphic
solutions of differential equations 
\cite{Blot-Mophou-N'Guerekata-Pennequin,Cieutat-Fatajou-N'Guerekata,
Fatajou-Minh-N'Guerekata-Pankov,Gal-N'Guerekata,Goldestein&Guerekata,
Liu-N'Guerekata-Minh,Minh-Naito-N'Guerekata,
N'Guerekata_1,N'Guerekata_2}. For some recent results of 
almost automorphic differential equations consult also
\cite{castillo,alan}.

In this paper,
we are initially motivated by a biological-mathematical 
model \cite{feng,grimmer,liz-martinez,Wei&Wang_1,Wei&Wang_2,zhao}
which is a delayed differential equation of the following type
\begin{eqnarray*}
y'(t)=-\delta(t) y(t)+p(t)g(y(t-\tau)), 
\end{eqnarray*}
with $\tau>0$, $\delta$ and $p$ almost automorphic functions
and $g$ a Lipschitz function. Then,  we focus our attention on 
the existence and uniqueness of solutions of the
following delayed differential equation
\begin{eqnarray}
y' &=& A(t)y+f(t)+g(t,y(t-\tau))
\qquad\mbox{with $\tau\geq 0$,}
\label{delay_lineal+f+g}
\end{eqnarray}
under several assumptions on $A,f$ and $g$. Naturally, the assumptions
on $A$ and $f$ are related to the almost automorphic behavior
and the assumptions on $g$ are mainly related with a  Lipschitz requirement.
We note that
\eqref{delay_lineal+f+g} naturally, includes as
particular cases the following equations
\begin{eqnarray}
y'&=&A(t)y,\label{lineal}\\
y'&=&A(t)y+f(t),\label{lineal+f}\\
y' &=& A(t)y+f(t)+g(t,y(t)).\label{lineal+f+g}
\end{eqnarray}
Thus, following a natural sequence of the classical systemic study of 
ordinary differential equations 
we start by analyzing the homogeneous linear equation \eqref{lineal}.
Then, we develop the theory for the non-homogeneous linear  equation \eqref{lineal+f}
by applying the method of variation of parameters .
In a third place, we analyze the nonlinear equation \eqref{lineal+f+g}
by using the fixed point arguments. Finally,
by a composition result of automorphic functions we get
an extension of the results for \eqref{lineal+f+g} to the 
delay equation \eqref{delay_lineal+f+g}. 

The main contributions 
and the organization of the paper are given as follows. In section~\ref{sec:preliminar}
we introduce the general assumptions, recall the concepts of almost
automorphicity and ergodic functions, define a convolution operator
and get the results for  \eqref{lineal}. To be a little more precise
in this section, we obtain some conditions for the exponential 
dichotomy  using ergodic functions and we prove the Bi-almost
periodicity and Bi-almost automorphicity of the Green function
when the evolution operator commute with the projection. We note that, the 
integral Bi-almost automorphicity property of the Green function
is fundamental to obtain the main results. In section~\ref{sec:mainres}, 
almost automorphicity of solutions 
of nonautonomous systems \eqref{lineal+f}, \eqref{lineal+f+g} 
and \eqref{delay_lineal+f+g} are obtained.
Here, the results of almost automorphicity of the differential equation
solutions are obtained by assuming that $A$ and $f$ are almost authomorphic and
$g$ satisfies \eqref{eq:condition_L}.
Finally, in section~\ref{sec:appl}
we study a biological model establishing an explicit condition 
under which there exist a unique almost automorphic solution 
of the Lasota-Wazewska equation.


\section{Preliminaries}
\label{sec:preliminar}

In this section we present some general assumptions, precise the concepts
related with the almost automorphic notion, we recall the notion
of ergodic functions, we define a convolution operator and the $\alpha$-exponential
dichotomy and introduce several results for the homogeneous equation  \eqref{lineal}
in the scalar, system and abstract case.

\subsection{General assumptions}
Here we present two general assumptions. Firstly,
throughout of the paper $(V,\Vert \cdot \Vert_{V})$ will be a Banach space
and let $(BC(\mathbb{R},V), \Vert \cdot \Vert_{\infty})$ will be used to denote
the Banach space of bounded continuous functions
from $\mathbb{R}$ into $V$ endowed with the sup 
norm $\Vert \varphi \Vert_{\infty}=\sup_{t\in \mathbb{R}} \Vert \varphi(t) \Vert_{V}$. 
Second, concerning to
the  assumptions on the coefficients $A,f$ and $g$ for  
equations \eqref{delay_lineal+f+g}-\eqref{lineal+f+g} we comment that
it will be specifically  done on the hypothesis of each result. However, 
in order to give a unified presentation,
we introduce some notation related to the assumption of
the local Lipschitz behavior of  $g$. 
Indeed, given a function $g$, it is assumed that:
\begin{eqnarray}
\left.
\begin{array}{cl}
\mbox{($g_0$)}&\mbox{$g(t,0)=0$ for all $t\in\mathbb{R}$;}\\
\mbox{($g_1$)}&\mbox{The function $g(t,y)$ is continuous on
$\mathbb{R}\times \Delta(\varphi_0,\rho)$ with $\Delta(\varphi_0,\rho)$
}
\\
&\mbox{the open ball centred in a given (fix) function $\varphi_0:\mathbb{R}\rightarrow V $
and with }
\\
&\mbox{radius $\rho\in\mathbb{R}^+$, i.e.,
$ \Delta(\varphi_0,\rho)=
 \Big\{ \varphi :\mathbb{R}\rightarrow V \Big|\;
\| \varphi - \varphi_0 \|_{\infty} \leq \rho\;
 \Big\}. 
 $ In par-}\\
 &\mbox{ticular, in subsections~\ref{subsec:lineal+f+g}-\ref{subsec:delay_lineal+f+g}
 will be assumed that  $\varphi_{0}$ is of the form}
\\
&\mbox{$\varphi_{0}(t)=\int_{\mathbb{R}} G(t,s)f(s)ds$
with $G$ the Green function defined on \eqref{green_c0}; }\\
\mbox{($g_2$)}&\mbox{There exist a positive constant $L$ such that
the inequality}\\
&
\mbox{$\| g(t,y_1)-g(t,y_2) \| \leq L \| y_1 -y_2 \|$ 
holds for all $(t,y_1,y_2)\in\mathbb{R}\times\Delta(\varphi_0,\rho)^2$.}
\end{array}
\right\}
\label{eq:condition_L}
\end{eqnarray}
This set of conditions ($g_0$)-($g_2$)
appear in several parts of the paper and essentially when we study the nonlinear
equations in subsections~\ref{subsec:lineal+f+g}-\ref{subsec:delay_lineal+f+g}.

\subsection{Almost automorphic notion and related concepts.}

We recall that the almost automorphic functions have been developed by 
Bochner \cite{Bochner_1,Bochner_2} as a generalization of almost
periodic functions. We recall that
a function $f\in BC(\mathbb{R},V)$ is called Bohr almost periodic 
\cite{corduneanu} if for each $\epsilon>0$, there
exist $l_{\epsilon}>0$ such that every interval of
length $l_{\epsilon}$ contains a number $\xi$ with the property: 
$\Vert f(t+\xi)-f(t) \Vert_{V} \leq \epsilon \  \mathrm{for } \  t\in \mathbb{R}.$
The set of Bohr almost periodic will be denoted by $BC(\mathbb{R},V)$. 
Then, we precise the concept of almost automorphic functions and matrices.

\begin{defn}
\label{def_automorphic}
Consider $V$ a Banach space. Then,
\begin{enumerate}[(i)]
 \item A continuous function $\psi:\mathbb{R}\rightarrow V$ is called an almost 
automorphic function if for any sequence of real numbers
$\{ \tilde{\tau}_n \}_{n=1}^{\infty}$,
there exist a subsequence 
$\{\tau_n \}_{n=1}^{\infty}$ of $ \{ \tilde{\tau}_n \}_{n=1}^{\infty}$ such
that the limit of the sequence $\{ \psi(t+\tilde{\tau}_n )\}_{n=1}^{\infty}$,
denoted by $\tilde{\psi}(t)$, is well defined for all $t\in\mathbb{R}$
and the sequence $\{ \tilde{\psi}(t-\tilde{\tau}_n )\}_{n=1}^{\infty}$
converges pointwise on $\mathbb{R}$ to $\psi(t)$, or equivalently
\begin{eqnarray}
\tilde{\psi}(t)=\lim_{n\rightarrow \infty } \psi(t+\tau_{n})
 \quad\mbox{and}\quad
 \psi(t)=\lim_{n\rightarrow \infty } \tilde{\psi}(t-\tau_{n})\label{tilde_psi}
\end{eqnarray}
are well defined for all $t\in \mathbb{R}$. 
The collection of all almost  automorphic functions from 
$\mathbb{R}$ to $V$ is denoted by $AA(\mathbb{R},V).$

\item
A matrix valued function $A:\mathbb{R}\to\mathbb{C}^{d_1\times d_2}$ 
is called an almost automorphic matrix valued function or equivalently
(most of the time by briefness) $A(t)\in \mathbb{C}^{d_1\times d_2}$
is called an almost automorphic matrix if for any 
sequence $\{\xi'_{n} \}_{n=1}^{\infty}\subset \mathbb{R}$, there
exist a subsequence $\{\xi_n\}_{n=1}^{\infty}$ of $ \{\xi'_n\}_{n=1}^{\infty}$ 
and a matrix $B(t)\in \mathbb{C}^{d_1\times d_2}$ 
such that the sequences $\{A(t+\xi_n)\}_{n=1}^{\infty}$
and $\{B(t-\xi_n)\}_{n=1}^{\infty}$ converges pointwise to $B(t)$
and $A(t)$, respectively.

\end{enumerate}

\end{defn}

We note that the convergence in \eqref{tilde_psi}
is pointwise. Then,   the function $\tilde{\psi}$ in \eqref{tilde_psi} 
is measurable, but not necessarily continuous. Moreover,
we note if we consider that convergence in definition \ref{def_automorphic} 
is uniform on $\mathbb{R}$ instead of pointwise convergence, 
we get that the function $\psi$ is  
Bochner almost periodic. It is well known that both definitions 
of almost periodicity (Bohr and Bochner)
are equivalents, see  for instance \cite{corduneanu}. 
Now, we note that  $AP(\mathbb{R},V)$ and 
$AA(\mathbb{R},V)$ are vectorial space 
and 
$AP(\mathbb{R},V) $ is a proper  subspace  of
$ AA(\mathbb{R},V)$, since for instance 
$\psi (t)= \cos \left( [2+\sin(t)+\sin(t\sqrt{2})]^{-1} \right) $ 
is an almost periodic function but not almost authomorphic.
Similarly, it is proven that the inclusion
$AP(\mathbb{R},V)\subset BC(\mathbb{R},V) $, for an extensive
discussion consult
\cite{caraballo-cheban,Bochner_2,Goldestein&Guerekata,
Gal-N'Guerekata,Liu-N'Guerekata-Minh,Minh-Naito-N'Guerekata,
Minh-Dat,N'Guerekata_1,N'Guerekata_2,Veech_1,Veech_2, 
Zaidman_1,Zaki,Ding&Xiao&Liang,Blot-Mophou-N'Guerekata-Pennequin,
Xiao-Zhu-Liang,Zaidman_2,Fatajou-Minh-N'Guerekata-Pankov}.

To close this subsection we introduce two additional facts.
Firstly, we note that
the simpler equation \eqref{lineal+f} with $A \equiv 0$, i.e., 
$y'(t)=f(t),$
with $f\in AA(\mathbb{R},V)$ has not necessarily a 
solution $y\in AA(\mathbb{R},V)$. However, this fact
is true in a uniformly convex Banach space 
$V$ and hence in every Hilbert space, see Theorem~\ref{1.1}. 
In the second place 
we need a composition result \cite{Cieutat-Fatajou-N'Guerekata},
which will be fundamental for the analysis of  \eqref{lineal+f+g}
and \eqref{delay_lineal+f+g}, see Proposition~\ref{composition}.

\begin{thm}\label{1.1}
Denote by $C_0$ the vectorial space formed 
by the functions which vanishes at infinity.
Consider that $V$ is a Banach space which does not contain $C_{0}$ 
as an isomorphic subspace and let $f\in AA(\mathbb{R},V)$. 
Then, the function $F(t)=\int_0^{t}f(s)ds$ is in $AA(\mathbb{R},V)$ 
if and only if it is bounded.\\
Such Banach space with this property for $F$ will be called a
Banach space with the Bohl-Bohr property.
\end{thm}

\begin{prop}\label{composition}
Let $g=g(t,y)\in AA(\mathbb{R}\times V,V)$  uniformly in $t$ for
$y$ in a compact set contained in $V$ and $g$ satisfies the assumptions
given on \eqref{eq:condition_L}. Then, $g(t,\varphi(t))\in AA(\mathbb{R},V)$ for all 
$\varphi\in AA(\mathbb{R},\Delta (\varphi_0,\rho))$.
\end{prop}

\subsection{Ergodic functions}
Here we introduce the concept of ergodic functions
and deduce that  these types  of functions implies naturally
an exponential behavior (or $\alpha$-exponential dichotomy to be more
precise).

\begin{defn}\label{mean}
A function $f\in BC(\mathbb{R},V)$ is called an ergodic function if the limit 
\begin{eqnarray*}
M(f) = \lim_{T\rightarrow \infty} \frac{1}{2T} \int_{-T+\xi}^{T+\xi}f(s)ds
\end{eqnarray*}
exists uniformly with respect to $\xi\in \mathbb{R}$ and its value is independent of $\xi$. 
The complex number $M(f)$ is called the mean of the function $f$.
\end{defn}

The mean of an  ergodic function has several properties, 
a complete list of properties
may be consulted in \cite{Zhang_ergodic}. 
Among these useful basic properties, we only recall 
the translation invariance property, since it
will be used frequently in the proofs given
below in this paper. Indeed,  the translation invariance property of 
$M(f)$, set  that $M(f)$ satisfies the following identity
\begin{eqnarray}
M(f)=M(f_{\xi}),\ \label{translation_invarance}
\end{eqnarray}
where $f_{\xi}$ denotes a $\xi$-translation of $f$, i.e. 
$f_{\xi}(t)=f(t+\xi)$ for all $t\in \mathbb{R}$ and any 
arbitrarily given $\xi \in \mathbb{R}$ (but fixed).

\begin{lem}\label{lema_mean}
Consider $\mu\in BC(\mathbb{R},\mathbb{C})$ 
an ergodic function with  $\mathrm{Re}(M(\mu))\neq 0$
and also consider $\alpha\in\mathbb{R}^+$. 
Then, there exist two positive constants 
$T_0$ (big enough) and $c$ such that
the two assertions given below are valid:
\begin{enumerate}[(i)]
\item \label{int_exponencial_negativo} If
$\mathrm{Re}(M(\mu))\in ]-\infty,-\alpha[\subset\mathbb{R}^-$, then
the following inequalities hold true:
\begin{eqnarray}
&&\mathrm{Re}\left( \int_{s}^{t}\mu(r)dr\right)
<-\alpha(t-s)
\quad 
\mbox{for}
\quad
t-s>T_0,\label{int_negativa}
\\
&&\left| \exp\left({\int_{s}^{t}\mu(r)dr}\right)\right | 
\leq
c\exp({-\alpha(t-s)})
\quad
\mbox{for all $(t,s)\in\mathbb{R}^2$ such that 
$t\geq s\geq 0$}.
\label{exp_negativa}
\end{eqnarray}

\item \label{int_exponencial_positivo}
If
$\mathrm{Re}(M(\mu))\in ]\alpha,\infty[\subset\mathbb{R}^+$, then
the following inequalities hold true:
\begin{eqnarray}
&&\mathrm{Re}\left( \int_{s}^{t}\mu(r)dr\right)
<\alpha(t-s)
\quad 
\mbox{for}
\quad
s-t>T_0,\label{int_positiva}
\\
&&\left| \exp\left({\int_{s}^{t}\mu(r)dr}\right)\right | 
\leq
c\exp({\alpha(t-s)})
\quad
\mbox{for all $(t,s)\in\mathbb{R}^2$ such that 
$s\geq t\geq 0$.}
\label{exp_postiva}
\end{eqnarray}
\end{enumerate}
\end{lem}

\begin{proof}
Let us assume that $\mu\in BC(\mathbb{R},\mathbb{C})$ is 
an ergodic function.  Then, by the Definition~\ref{mean}
and the translation invariance property of $M(f)$ 
(see \eqref{translation_invarance}) we have that
\begin{eqnarray}
 \int_{0}^{T}\mu(t+\tau)d\tau=[M(\mu)+o(1)]T
 \quad\mbox{when}\quad T\rightarrow \infty.
\end{eqnarray}
Here and throughout of the paper $o(1)$ corresponds to the
well known Bachmann-Landau notation, i.e.
$f=o(g)$ if and only if $(f/g)(x)\to 0$ when $x\to \infty.$
Then, when $T=t-s$, we get
\begin{eqnarray}
\int_{s}^{t}\mu(r)dr=\int_{0}^{T} \mu(-s+\tau)d\tau
\end{eqnarray}
and the proof of \eqref{int_negativa}  follows immediately.
The proof of \eqref{int_exponencial_negativo}
is a consequence of the exponential function increasing behavior.
Thus, the proof of item {\it (i)} is completed. Now,
the proof of item {\it (ii)} is similar
and we omit it.
\end{proof}

\subsection{The convolution operator}
Let us denote by $L^1(\mathbb{R})$ and
$L^\infty(\mathbb{R})$ the spaces 
of Lebesgue integrable functions on $\mathbb{R}$
and essentially bounded functions on $\mathbb{R}$, respectively. Then, 
the convolution operator on  
$\mathcal{L}:L^\infty(\mathbb{R})\to L^\infty(\mathbb{R})$ is defined as 
the operator such that
\begin{eqnarray}
\mathcal{L}(\varphi)(t)=\int_{\mathbb{R}} h(t-s)\varphi(s)ds,
\quad 
h\in L^{1}(\mathbb{R}),
\quad
t\in \mathbb{R},
\label{convolution_operator}
\end{eqnarray}
for all $\varphi \in L^\infty(\mathbb{R})$.
Some properties of $\mathcal{L}$, which will be needed in the proof
the main results, are summarized in the following lemma.

\begin{lem}\label{2.2}
Consider $\mathcal{L}$ 
the convolution operator  defined by \eqref{convolution_operator}.
Then, the spaces $BC(\mathbb{R}),$ $AP(\mathbb{R})$
and $AA(\mathbb{R})$ are invariants under the operator $\mathcal{L}$.
Moreover, the inequalities
\begin{eqnarray}
\|\mathcal{L}\varphi\|_{L^\infty(\mathbb{R})} 
&\le & \|\varphi\|_{L^\infty(\mathbb{R})} \|h\|_{L^1(\mathbb{R})}
\quad \mbox{for  $\varphi\in BC(\mathbb{R})$,}
 \label{cerrado_BC}
 \\
 \|(\mathcal{L}\varphi)_{\xi}-\mathcal{L}\varphi\|_{L^\infty(\mathbb{R})} 
&\le & \|(\varphi)_{\xi}-\varphi\|_{L^\infty(\mathbb{R})} \|h\|_{L^1(\mathbb{R})}
\quad \mbox{for $\xi\in\mathbb{R}$ and $\varphi\in AP(\mathbb{R})$,}
\label{cerrado_AA}
\end{eqnarray}
hold. Here $(\mathcal{L}\varphi)_{\xi}$ and $(\varphi)_{\xi}$
are the $\xi$-translation functions for 
$\mathcal{L}\varphi$ and $\varphi$, 
respectively.
\end{lem}
\begin{proof}
Let us select $\varphi\in AA(\mathbb{R})$. Then, by Definition~\ref{def_automorphic},
given an arbitrary sequence of real numbers
$\{ \tilde{\tau}_n \}_{n=1}^{\infty}$,
there exist a subsequence 
$\{\tau_n \}_{n=1}^{\infty}$ of $ \{ \tilde{\tau}_n \}_{n=1}^{\infty}$ such
that \eqref{tilde_psi} is satisfied.
Now, if we consider $\psi=L(\varphi)$, we have that 
\eqref{tilde_psi} is equivalently rewritten as follows
\begin{eqnarray}
\tilde{\psi}(t) = \lim_{n\rightarrow \infty} \psi(t+\tau_n)
\quad
\mbox{and}
\quad
\psi(t) =\lim_{n\rightarrow \infty} \tilde{\psi}(t-\tau_n).
\end{eqnarray}
Indeed, this fact can be proved by application of 
Lebesgue's dominated convergence theorem, since
\begin{eqnarray*}
\tilde{\psi}(t) 
&=& \lim_{n\rightarrow \infty} \mathcal{L}(\varphi_{\tau_n})(t)
	= \lim_{n\rightarrow \infty} \int_{\mathbb{R}}h(r)\varphi_{\tau_{n}}(t-r)dr
\\
&=& \int_{\mathbb{R}}h(r)\tilde{\varphi}(t-r)dr
	= \mathcal{L}(\tilde{\varphi})(t).
\end{eqnarray*}

Let us consider $\varphi\in BC(\mathbb{R})$, then 
we deduce \eqref{cerrado_BC} by application of the H\"older inequality.
Now, from \eqref{cerrado_BC} we follow the invariance
of $BC(\mathbb{R})$.
\end{proof}

\vspace{0.5cm}
We note
 that, if we define
 \begin{eqnarray*}
  h_1(x)=
  \left\{
  \begin{array}{lll}
   \exp(-\alpha x),&\quad&x>0,\\
   0,&&\mbox{otherwise},
  \end{array}
  \right.
  \quad
  \mbox{and}
  \quad
   h_2(x)=
  \left\{
  \begin{array}{lll}
   \exp(\alpha x),&\quad&x<0,\\
   0,&&\mbox{otherwise},
  \end{array}
  \right. 
 \end{eqnarray*}
 for some $\alpha\in\mathbb{R}^+$
 and denote by $\mathcal{L}_{1}$
and $\mathcal{L}_{2}$ the corresponding convolution operators associated
with $h_1$ and $h_2$, respectively.
Then, we get an interesting result by application of Lemma~\ref{2.2}.
More precisely, we have the following Corollary.

\begin{cor}
Consider $\alpha\in\mathbb{R}^+$ and 
the operators $\mathcal{L}_{i},\ i=1,2$ defined by
\begin{eqnarray}
\mathcal{L}_{1}(\varphi)(t)
	=\int_{-\infty}^{t}e^{-\alpha(t-s)}\varphi(s)ds
\quad\mbox{and}\quad
\mathcal{L}_{2}(\varphi)(t)
	=\int_{t}^{\infty}e^{\alpha(t-s)}\varphi(s)ds,
\label{eq:operators_exp_green}
\end{eqnarray}
respectively.
Then, the spaces $BC(\mathbb{R}),$ $AP(\mathbb{R})$
and $AA(\mathbb{R})$ are invariants under the operator $\mathcal{L}_{i},\ i=1,2$.
Moreover, the inequalities \eqref{cerrado_BC} and  \eqref{cerrado_AA} are satisfied
with  $\mathcal{L}_{i},i=1,2,$ instead of~$\mathcal{L}$.
\end{cor}

\subsection{Some concepts and properties related to equation \eqref{lineal}}
In this subsection we study the equation \eqref{lineal}. In order to introduce the 
concepts and results, we recall the standard notation of 
fundamental matrix and flow associated
with \eqref{lineal}, which are denoted by $\Phi_A$
and $\Psi_A$, respectively. More precisely 
\begin{eqnarray}
\left.
\begin{array}{l}
\mbox{Given a matrix $A(t)$, then the notation $\Phi_A=\Phi_{A} (t)$
and $\Psi_A$ are used for  
}
\\
\mbox{a  fundamental matrix of the system \eqref{lineal} 
and for the application defined}
\\
\mbox{as follows $\Psi_A(t,s)=\Phi_{A} (t)\Phi_{A}^{-1}(s)$  .}
\end{array}
\right\}
 \label{notation_linal}
\end{eqnarray}

\begin{lem}
\label{lema_identidad}
Consider the notation  \eqref{notation_linal}. Then,
the identities
\begin{eqnarray}
\Psi_{A}(t,s)-\Psi_{B}(t,s)
&=&\int_{s}^{t}\Psi_{B}(t,r)[A(r)-B(r)]\Psi_{A}(r,s)dr,
\label{identidad_lem}
\\
\Psi_{A}(t+\xi,s+\xi)-\Psi_{B}(t,s)
&=&\int_{s}^{t}\Psi_{B}(t,r)[A(r+\xi)-B(r)]\Psi_{A}(r+\xi, s+\xi)dr,
\label{identidad_cor}
\end{eqnarray}
are satisfied for all $(t,s,\xi)\in \mathbb{R}^3.$
\end{lem}

\begin{proof} 
Let us denote
by $H$ the function defined by the following correspondence rule
$H(t,s)=\Psi_{A}(t,s)-\Psi_{B}(t,s)$. 
Then, by partial differentiation of $H$ with respect to the first variable
and making some rearrangements, we get
\begin{eqnarray*}
\frac{\partial}{\partial t}H(t,s)
&=&\frac{\partial}{\partial t}\Psi_{A}(t,s)
	-\frac{\partial}{\partial t}\Psi_{B}(t,s)
\\
&=& (A(t)-B(t))\Psi_{A}(t,s)+B(t)(\Psi_{A}(t,s)-\Psi_{B}(t,s))
\\
&=& (A(t)-B(t))\Psi_{A}(t,s)+B(t)H(t,s).
\end{eqnarray*}
Now, by multiplying to the left by $\Psi_{B}(t,r)$ and simplifying, we deduce that
\begin{eqnarray*}
\Psi_{B}(t,r)(A(r)-B(r))\Psi_{A}(r,s)
&=&\Psi_{B}(t,r)\frac{\partial}{\partial r}H(r,s)-\Psi_{B}(t,r)B(r)H(r,s)\\
&=&\Psi_{B}(t,r)\frac{\partial}
    {\partial r}H(r,s)-\left(\frac{\partial}{\partial r}\Psi_{B}(t,r)\right)H(r,s)\\
&=&\frac{\partial}{\partial r}\left(\Psi_{B}(t,r)H(r,s) \right).
\end{eqnarray*}
Thus, by integration over the interval $[s,t]$, we have that
\begin{eqnarray*}
\int_{s}^{t}\Psi_{B}(t,r)[A(r)-B(r)]\Psi_{A}(r,s)dr
&=&\Psi_{B}(t,t)H(t,s)-\Psi_{B}(t,s)H(s,s), 
\end{eqnarray*}
which implies \eqref{identidad_lem} by noticing that $\Psi_{B}(t,t)=I$ and
$H(s,s)=0$. Now, the proof \eqref{identidad_cor} follows by
similar arguments or by direct application of \eqref{identidad_lem}.
\end{proof}

\begin{lem}
\label{bi_ap}
Consider the notation  \eqref{notation_linal} and the sets 
$ \overrightarrow{\mathbb{R}}^2$ and $ \overleftarrow{\mathbb{R}}^2$ 
defined as follows
\begin{eqnarray*}
\overrightarrow{\mathbb{R}}^2=\Big\{(s,t)\in\mathbb{R}^2\;:\; s>t\Big\}
\quad
\mbox{and}
\quad
\overleftarrow{\mathbb{R}}^2=\Big\{(s,t)\in\mathbb{R}^2\;:\; s<t\Big\},
\end{eqnarray*}
respectively.
Assume that the following three statements are true:
$A(t)$ is an almost periodic matrix (see definition~\ref{def_automorphic}),
$P$ is a constant projection matrix that commutes with
$\Phi_A$ and for 
some given positive constants $c$ and $\alpha$ the inequality
\begin{eqnarray}
\| \Psi_A(t,s)P \| \leq c\;\exp(-\alpha |t-s|),
\label{dicotomia_Phi_P}
\end{eqnarray}
is satisfied for all $(t,s)\in \overrightarrow{\mathbb{R}^2}$
\Big(or for all $(t,s)\in \overleftarrow{\mathbb{R}^2}$\Big).
Then, for all $(t,s,\xi)\in \overrightarrow{\mathbb{R}^2}\times\mathbb{R}$
\Big(respectively $(t,s)\in \overleftarrow{\mathbb{R}^2}\times\mathbb{R}$\Big), 
there exist two real constants
$c_1>0$ and $\alpha'\in ]0,\alpha[$, such that
\begin{eqnarray}
\| \Psi_A(t+\xi ,s+\xi)P-\Psi_A(t,s)P \| \leq 
c_1 \Vert A(\cdot +\xi)-A(\cdot)\Vert_{\infty} 
\exp({-\alpha '\vert t-s\vert}). 
\label{des_A}
\end{eqnarray}
In particular, if $\xi$ is an $\epsilon$-almost period of $A$ the inequality
\begin{eqnarray}
\| \Psi_A(t+\xi ,s+\xi)P-\Psi_A(t ,s)P \| 
\leq c_1\;\epsilon \exp(-\alpha' \vert t-s \vert), 
\label{bi_phiP}
\end{eqnarray}
is satisfied for all 
$(t,s,\xi)\in \overrightarrow{\mathbb{R}^2}\times\mathbb{R}$
\Big(respectively $(t,s)\in \overleftarrow{\mathbb{R}^2}\times\mathbb{R}$\Big)
or equivalently $\Psi_A(t,s)P$ is Bi-almost periodic.
\end{lem}

\begin{proof}
Since the proofs with $(t,s,\xi)\in \overrightarrow{\mathbb{R}^2}\times\mathbb{R}$
or with $(t,s,\xi)\in \overleftarrow{\mathbb{R}^2}\times\mathbb{R}$
are analogous, we consider only one of the cases. Then, 
in order to fix ideas, let us consider 
$(t,s,\xi)\in \overrightarrow{\mathbb{R}^2}\times\mathbb{R}$ and
recall the notation $\Delta_{\xi}$ defined by
\begin{eqnarray}
\mbox{$\Delta_{\mathbf{\xi}}F(\mathbf{x})
=F(\mathbf{x}+\mathbf{\xi})-F(\mathbf{x})$
for any function $F$}. 
\label{shif_notation}
\end{eqnarray}
In particular, for instance, we have that
$\Delta_{\xi}\Phi(t,s)=\Phi(t+\xi,s+\xi)-\Phi(t,s)$
and $\Delta_{\xi}A(t)=A(t+\xi)-A(t)$.
From \eqref{identidad_lem} and the hypothesis that
$P$ is a constant projection matrix, i.e. $P^{2}=P$, 
which commutes with $\Phi(t)$ for every $t\in \mathbb{R}$,
we have that
\begin{eqnarray}
\Delta_{\xi}\Psi_{A}(t,s)P
=\int_{s}^{t}\Psi_{A}(t,r)P\Delta_{\xi}A(r)\Psi_{A_\xi}(r,s)Pdr.
\label{identidad_projection}
\end{eqnarray}
Now,  by the assumption \eqref{dicotomia_Phi_P}
we follow that 
\eqref{identidad_projection}
implies the following estimate
\begin{eqnarray*}
\|\Delta_{\xi}\Psi_{A}(t,s)P\|
&\le& c\;\exp(-\alpha |t-s|)
\int_{s}^{t}\|\Delta_{\xi}A(r)\|dr
\\
&\le& c\;\exp(-\alpha |t-s|)\;
\|\Delta_{\xi}A\|_{\infty}\;|t-s|,
\end{eqnarray*}
which implies \eqref{des_A}. The inequality \eqref{bi_phiP} follows
immediately from \eqref{des_A} using the fact that $\xi$
is an $\epsilon$-almost period of $A$.
\end{proof}

\begin{defn}
\label{intregal_bi_AA}
Consider the notation  \eqref{notation_linal}.
If the fact that $A(t)$ is an almost automorphic matrix
(see definition~\ref{def_automorphic}-(ii))
implies the following convergence
\begin{eqnarray}
\lim_{n\to\infty}
\int_{-\infty}^{t} \left\| \Big(\Psi_{A_{\xi_{n}}}-\Psi_{B}\Big)(t,s)P  \right\| ds
=
\lim_{n\to\infty}
\int^{t}_{-\infty}\left\| \Big(\Psi_{B_{-\xi_{n}}}-\Psi_{A}\Big)(t,s)P  \right\| ds
=0,
\label{exp_phi_aa}
\end{eqnarray}
the application $\Psi_{A}$ is called integrally Bi-almost automorphic
on $]-\infty,t]$.
Similarly, if the almost automorphic behavior of $A(t)$ 
implies the following convergence
\begin{eqnarray}
\lim_{n\to\infty}
\int_{t}^{\infty} \left\| \Big(\Psi_{A_{\xi_{n}}}-\Psi_{B}\Big)(t,s)P  \right\| ds
=
\lim_{n\to\infty}
\int_{t}^{\infty}\left\| \Big(\Psi_{B_{-\xi_{n}}}-\Psi_{A}\Big)(t,s)P  \right\| ds
=0,
\label{exp_phi_bbbb}
\end{eqnarray}
the application $\Psi_{A}$ is called integrally Bi-almost automorphic
on $[t,\infty[$.
Here  $\{\xi_{n} \}_{n=1}^{\infty}$ and $B(t)$ denotes the subsequence and
the matrix with the properties given on definition~\ref{def_automorphic}-(ii).
\end{defn}

\begin{lem}
\label{bi_aa}
Consider the notation  \eqref{notation_linal}.
If the assumptions of Lemma~\ref{bi_ap} hold, then
$\Psi_A$ is integrally Bi-almost automorphic
on $]-\infty,t]$ and on $[t,\infty[$.
\end{lem}

\begin{proof}
Let us consider that $A(t)$ is an almost automorphic matrix.
Then by definition~\ref{def_automorphic}-(ii) we follow that
for any sequence
$\{\xi'_{n} \}_{n=1}^{\infty}\subset \mathbb{R}$, there
exist a subsequence $\{\xi_n\}_{n=1}^{\infty}$ of $ \{\xi'_n\}_{n=1}^{\infty}$ 
and a matrix $B(t)$ such that 
\begin{eqnarray}
 \lim_{n\to\infty} A_{\xi_n}(t)=B(t)
 \quad
 \mbox{and}
 \quad
 \lim_{n\to\infty} B_{-\xi_n}(t)=A(t)
 \quad
 \mbox{for all $t\in\mathbb{R}$}.
 \label{lem:bi_alm_aut:1}
\end{eqnarray}
Now, from the identity \eqref{identidad_lem}, 
we deduce that the assumption \eqref{dicotomia_Phi_P} implies
the following bounds
\begin{eqnarray}
\left\| \Big(\Psi_{A_{\xi_{n}}}-\Psi_{B}\Big)(t,s)P\right\|
&\leq& ce^{-\alpha\vert t-s\vert} 
\left\| \int_{s}^{t}  (A_{\xi_{n}}-B)(r)  dr \right\|
\nonumber\\
&\leq& c_1 e^{-\alpha' \vert t-s\vert}(\Vert A\Vert+\Vert B\Vert),
\label{lem:bi_alm_aut:2}
\\
\left\| \Big(\Psi_{B_{-\xi_{n}}}-\Psi_{A}\Big)(t,s)P\right\|
&\leq& ce^{-\alpha\vert t-s\vert} 
\left\| \int_{s}^{t}  (B_{-\xi_{n}}-A)(r)  dr \right\|
\nonumber\\
&\leq& 
c_1 e^{-\alpha' \vert t-s\vert}(\Vert A\Vert+\Vert B\Vert),
\label{lem:bi_alm_aut:2_2}
\end{eqnarray}
for all $(t,s,n)\in\mathbb{R}^2\times\mathbb{N}$ and some real constants $c_1>0$
and $\alpha'\in ]0,\alpha[$.
Then, by applying four times the Lebesgue's dominated convergence theorem 
we deduce the integrally Bi-almost automorphic property
of $\Psi_A$. Indeed, firstly by the first limit given
in \eqref{lem:bi_alm_aut:1}
we deduce that for each $(s,t)\in\mathbb{R}^2$ the integral
$\int_{s}^{t} \| (A_{\xi_{n}}-B)(r) \| dr$
 converges to $0$  when $n\to\infty$. 
Then, with this convergence in mind,
in a second application, from \eqref{lem:bi_alm_aut:2}
we get that for each $t\in \mathbb{R}$ the integral
$\int_{-\infty}^{t} \left\| (\Psi_{A_{\xi_{n}}}-\Psi_{B})(t,s)P  \right\| ds$
converges to $0$  when $n\to\infty$.
Similarly,  applying twice more the Lebesgue's theorem
(in the second integral \eqref{lem:bi_alm_aut:1} and in 
the inequality \eqref{lem:bi_alm_aut:2_2}) we deduce that
for each $(s,t)\in\mathbb{R}^2$ the integral
$\int_{s}^{t} \| (B_{-\xi_{n}}-A)(r) \| dr$ converges
to  $0$ when $n\to\infty$
and for each $t\in\mathbb{R}$ the integral
$\int_{-\infty}^{t} \left\| (\Psi_{B_{-\xi_{n}}}-\Psi_{A})(t,s)P  \right\| ds$
converges to $0$ when $n\to\infty$.
Thus, \eqref{exp_phi_aa} holds and $\Psi_A$ is integrally Bi-almost automorphic.
The proof of \eqref{exp_phi_bbbb} can be obtained by similar arguments.
\end{proof}

\begin{defn}
\label{def_dichotomy}
Consider the notation \eqref{notation_linal}.
The linear system \eqref{lineal} has an 
$\alpha$-exponential dichotomy if there exist
a projection $P$ and two positive  constants 
$c$ and $ \alpha$
such that for all $ (t,s)\in \mathbb{R}^2$ the estimate 
\begin{eqnarray}
\| G_A(t,s)\|\leq c\exp({-\alpha\vert t-s\vert})
\quad\mbox{with}\quad
G_A(t,s)=\begin{cases}
\Phi_A(t)P\Phi^{-1}_A(s), &\mbox{$t\geq s$,} \\
-\Phi_A(t)(I-P)\Phi^{-1}_A(s), &\mbox{otherwise,}
\end{cases}\label{green}
\end{eqnarray} 
is satisfied. The matrix $G_A$ is called the Green matrix associated with the dichotomy.
\end{defn}

\begin{lem}\label{bi_aa_integrable}
Consider the notation  \eqref{notation_linal} and 
define the Green operator $\Gamma$ as follows
\begin{eqnarray*}
(\Gamma\varphi)(t) &=& \int_{-\infty}^{\infty} G(t,s)\varphi(s)ds,\ t\in \mathbb{R}.
\end{eqnarray*}
Assume that $A(t)$ is an almost automorphic matrix and 
\eqref{lineal} has an exponentially 
dichotomy  such the its projection commutes with the fundamental matrix $\Phi_A$. 
Then, the following assertions are satisfied:
\begin{enumerate}[(i)]
 \item The Green matrix $G_A$ is integrally Bi-almost automorphic.
 \item The spaces $BC(\mathbb{R},V),$ $AP(\mathbb{R},V)$
and $AA(\mathbb{R},V)$ are invariants under the operator $\Gamma$.
Moreover, there exist two positive constants $c_1$ and $c_2$
such that the following inequalities
\begin{eqnarray*}
&& \| \Gamma \varphi\|_{\infty}\leq \frac{\| \varphi \|_{\infty}}{\alpha}
\quad\mbox{for $\varphi\in BC(\mathbb{R},V)$,}
\\
&&
\| (\Delta_{\xi}\Gamma\varphi)(t) \|
\leq  c_1\| \varphi\|_{\infty} \sum_{i=1}^2 
\mathcal{L}_i\left(\vert\Delta_{\xi}A \vert\right) 
+c_2\sum_{i=1}^2 \mathcal{L}_i \left( \vert\Delta_{\xi}\varphi \vert\right)
\quad\mbox{for $\varphi\in AP(\mathbb{R},V)$,}
\end{eqnarray*}
are satisfied.
Here $\mathcal{L}_i$ and $\Delta_{\xi}$ denotes the operators defined
on \eqref{eq:operators_exp_green} and \eqref{shif_notation}, respectively.
\end{enumerate}

\end{lem}

\begin{proof}
The proofs of {\it (i)} and {\it (ii)} are straightforward.
Indeed, for {\it(i)},
let us consider $A(t),B(t)$ and $\{\xi_n\}_{n=1}^\infty$
as given in the proof of Lemma~\ref{bi_aa}. Now,
by the assumptions we can deduce that 
\begin{eqnarray*}
\int_{-\infty}^{\infty}\| (G_{A_{\xi_n}}-G_{B})(t,s) \| ds\rightarrow 0
\quad
\mbox{and}
\quad
\int_{-\infty}^{\infty}\| (G_{B_{-\xi_n}}-G_{A})(t,s) \| ds\rightarrow 0 
\quad
\mbox{ when $n\rightarrow \infty$}
\end{eqnarray*}
where $G_{A}$ is the Green matrix defined on \eqref{green}. Thus, we can follow
that $G_{A}$ is integrally Bi-almost automorphic. 
Meanwhile, we follow the proof of {\it (ii)}
by application of Lemmas~\ref{lema_mean},
 \ref{2.2} and \ref{bi_aa_integrable}-{\it(i)}.
\end{proof}

\section{Main Results}
\label{sec:mainres}

In this section we present several results of Massera
type for \eqref{lineal+f} and \eqref{delay_lineal+f+g}
and related with the almost authomorphic behavior of the
$A$ and $f$.

\subsection{Results for \eqref{lineal+f}}
Here we present a result for the  scalar abstract case,
see Theorem~\ref{massera}.
Then, we extend this result can to linear triangular systems and
general linear constant systems, see Theorem~\ref{massera_matrix}. 
We also, present simple and useful relation between finite and infinite dimension
is deduced from Theorem~\ref{teo:fin_vs_infin}. Finally,
we present to results on the general
case , see Theorems~\ref {teo:SoL_lineal+f_abstract}
and \ref{cor_unica_solucion}.

\begin{thm}
\label{massera}
Consider the equation \eqref{lineal+f} with $A=\mu:\mathbb{R}\to\mathbb{C}$
and denote by $g$ the application defined by
\begin{eqnarray}
g(s,t)=\exp\Big({\int_{s}^{t}\mu(r)dr}\Big).
 \label{g:caso_escalar}
\end{eqnarray}
Then, the following assertions are valid:
\begin{enumerate}[(i)]
\item \label{massera_i} 
Assume that  $\mu$ is a function belongs to $ AA(\mathbb{R},\mathbb{C})$
satisfying  $M(\mathrm{Re}(\mu))\neq 0$.
Also, assume that $f$ is belongs to $AA(\mathbb{R},V)$. Then, a
solution $y$ of equation \eqref{lineal+f} is bounded 
if and only if $y\in AA(\mathbb{R},V)$
or equivalently the unique solution of equation \eqref{lineal+f}
belongs $ AA(\mathbb{R},V)$ is given by
\begin{eqnarray}
y(t)=
\left\{
\begin{array}{lll}
{\displaystyle\int_{-\infty}^{t}} g(s,t)f(s)ds,
	&&M(\mathrm{Re}(\mu))<0,
	\\
 -{\displaystyle\int_{t}^{\infty}} g(s,t)f(s)ds,
	&&M(\mathrm{Re}(\mu))>0.
\end{array}
\right.
\label{sol_heter_partes}
\end{eqnarray}

\item Assume that
$\mu(t)=ia(t)$ with $\int_{0}^{t} a(s)ds$ bounded and 
$V$ a Bohl-Bohr Banach space. Then, a solution 
$y$ of equation \eqref{lineal+f}
is bounded if and only if $y$ is belongs $AA(\mathbb{R},V)$ and is given by
\begin{eqnarray}
y(t)=\exp\Big({i\int_{0}^{t}a(r)dr}\Big)v
+\int_{0}^{t}\exp\Big({i\int_{s}^{t} a(r)dr}\Big)f(s)ds
\quad
\mbox{for all $v\in V$}. 
\end{eqnarray}
\end{enumerate}
\end{thm}

\begin{proof} {\it(i)} Before of  prove the item
we deduce two estimates (see \eqref{inequality_cor} and \eqref{inequality_cor_2}) 
and introduce some notation (see \eqref{aa_mu} to \eqref{aa_f}). 
Firstly, by Lemma~\ref{lema_mean} we can deduce that
the scalar equation  
$x'=\mu(t)x$ has an $\alpha$-exponential dichotomy. Indeed,
we note that
by the hypothesis $M(\mathrm{Re}(\mu))\not=0$ we can
always
select $\alpha$ satisfying $\vert M(\mathrm{Re}(\mu))\vert \geq \alpha>0$.
Then, by application of Lemma~\ref{lema_mean}, we have that
there exist a positive constant $c$ such that
\begin{eqnarray}
| g(t,s)|\leq ce^{-\alpha| t-s|}
\quad
\mbox{with $g$  defined in \eqref{g:caso_escalar}. }
\label{inequality_green}
\end{eqnarray}
Moreover, by application of Lemma \ref{bi_ap}
and integration on $s$, we deduce
that there exist $c_1\in\mathbb{R}^+$ and $\alpha'\in ]0,\alpha[$
such the following
inequalities
\begin{eqnarray}
&&\int_{-\infty}^{t} |\Delta_{\xi}g(t,s)|\;|f_{\xi}(s)| ds
\leq \frac{c_1}{\alpha'}  \| f\|_\infty\|\Delta_{\xi}\mu\|_\infty
\quad
\mbox{for all $t\in\mathbb{R}$,} 
\label{inequality_cor}
\\
&&\int_{t}^{\infty} |\Delta_{\xi}g(t,s)|\;|f_{\xi}(s)| ds
\leq \frac{c_1}{\alpha'}  \| f\|_\infty\|\Delta_{\xi}\mu\|_\infty
\quad
\mbox{for all $t\in\mathbb{R}$,} 
\label{inequality_cor_2}
\end{eqnarray}
hold. Now, let us consider
$\{\xi_{n} \}_{n=1}^\infty$ a sequence in $\mathbb{R}$.
Then, by the assumption $\mu$  belongs $ AA(\mathbb{R},V)$, 
there exist a subsequence $\{\xi'_{n} \}_{n=1}^\infty$ 
of
$\{\xi_{n} \}_{n=1}^\infty$ and the function $\tilde{\mu}$ 
such that
\begin{eqnarray}
\lim_{n\rightarrow \infty} \mu_{\xi'_n}(t)=\tilde{\mu}(t),
\quad
\lim_{n\rightarrow \infty} \tilde{\mu}_{-\xi'_n}(t)=\mu(t),
\quad\mbox{for all $t\in\mathbb{R}$}. 
\label{aa_mu}
\end{eqnarray}
Similarly, given the sequence $\{\xi'_{n} \}_{n=1}^\infty$
by the hypothesis $f\in AA(\mathbb{R},V)$, 
there exist a subsequence $\{\xi''_{n} \}_{n=1}^\infty$  of
$\{\xi'_{n} \}_{n=1}^\infty$ and the function $\tilde{f}$ 
such that
\begin{eqnarray}
&&\lim_{n\rightarrow \infty} f_{\xi''_n}(t)=\tilde{f}(t),
\quad
\lim_{n\rightarrow \infty} \tilde{f}_{-\xi''_n}(t)=f(t), 
\quad\mbox{for all $t\in\mathbb{R}$}.
\label{aa_f}
\end{eqnarray}

Now we develop the proof of the item. Indeed, we consider that $y$
is defined by \eqref{sol_heter_partes} and we prove that
$y$ is belongs to $AA(\mathbb{R},V)$.
Let us start by considering the notation
$y_{\pm}$ and $\tilde{y}_{\pm}$
for the functions defined as follows
\begin{eqnarray}
&&
y_+(t)=\int_{-\infty}^{t}g(t,s)f(s)ds,
\quad
\tilde{y}_+(t)=\int_{-\infty}^{t}g(t,s)\tilde{f}(s)ds,
\label{notacion_ad_0}\\
&&
y_{-}(t)=-\int_{t}^{\infty}g(t,s)f(s)ds
\quad
\mbox{and}
\quad
\tilde{y}_{-}(t)=-\int_{t}^{\infty}g(t,s)\tilde{f}(s)ds,
\label{notacion_ad}
\end{eqnarray}
respectively.
Then, by algebraic rearrangements we deduce that
\begin{eqnarray}
(y_+)_{\xi''_{n}}(t)-\tilde{y}_+(t)
&=&\int_{-\infty}^{t}\Delta_{\xi''_{n}}g(t,s) f_{\xi''_{n}}(s)ds 
+\int_{-\infty}^{t}g(t,s)(f_{\xi''_{n}}-\tilde{f})(s)ds,
\qquad
\label{see_cauchy}
\\
(y_-)_{\xi''_{n}}(t)-\tilde{y}_-(t)
&=&-\int_{t}^{-\infty}\Delta_{\xi''_{n}}g(t,s) f_{\xi''_{n}}(s)ds 
-\int_{t}^{-\infty} g(t,s)(f_{\xi''_{n}}-\tilde{f})(s)ds,
\label{see_cauchy_2}
\end{eqnarray}
Now, by using Lebesgue dominated convergence theorem we get that
the four integrals on \eqref{see_cauchy}-\eqref{see_cauchy_2}
converges to $0$ when $n\to\infty$.
Indeed, by \eqref{inequality_cor} and \eqref{aa_mu} we follow that 
the first integral in  \eqref{see_cauchy} converges 
to $0$ when $n\rightarrow \infty$. We see that
the second integral in \eqref{see_cauchy} vanishes when $n\to \infty$
by consequence of \eqref{aa_f}.   Meanwhile, we note that both integrals in 
\eqref{see_cauchy_2} converge to $0$ when $n\to \infty$
by application of \eqref{inequality_cor_2}, \eqref{aa_mu} and \eqref{aa_f}.
Consequently, we have that
\begin{eqnarray}
 \lim_{n\rightarrow \infty} (y_{\pm})_{\xi_{n}}(t)=\tilde{y}_{\pm}(t)
 \quad
 \mbox{for all $t\in\mathbb{R}$.}
\end{eqnarray}
Similarly, we can prove that $(\tilde{y}_{\pm})_{-\xi_{n}}(t)\to y(t)$
for all $t\in\mathbb{R}$ and when $n\to\infty$.
Hence $y\in AA(\mathbb{R},V)$.

\vspace{0.5cm}
\noindent
{\it(ii)}
Noticing that $h(s)=\exp{(i\int_{0}^{s}a)}f(s)\in AA(\mathbb{R},V)$ and the Banach
space $V$
has the Bohl-Bohr property we follow the proof by application of Theorem~\ref{1.1}.
\end{proof}

\begin{thm}
\label{massera_matrix}
Consider that $A(t)\in AA(\mathbb{R},\mathbb{C}^{p\times p})$ 
is an upper triangular of order $p \times p$.
Then the following assertions are valid:
\begin{enumerate}[(i)]
 \item Assume that $A(t)$ satisfies the condition
 $\mathrm{Re}(M(a_{ii}))\neq 0$ for all $i=1,\ldots,p$
 and $f\in AA(\mathbb{R},V^p)$. Then, 
a solution $y$ of \eqref{lineal+f} is bounded
if and only if  $y\in AA(\mathbb{R},V^{p})$
\item Assume that $A(t)$ satisfies the condition
\begin{eqnarray}
a_{kk}(t)=i\beta_{k}(t)
\quad\mbox{with}\quad 
\int_{0}^{t}\beta_{k}(s)ds
\quad\mbox{bounded for all $k=1,\cdots,p.$}
\end{eqnarray}
Assume that the Banach space $V$ has the Bohl-Bohr property.
Then any  solution $y$ of system \eqref{lineal+f} is bounded
if and only if  $y$ is belongs $AA(\mathbb{R},V^{p})$.

\end{enumerate}
\end{thm}

\begin{proof}
{\it (i)}
If we consider that  $A(t)$ is triangular matrix, we have that
the system \eqref{lineal+f} is of the following type
\begin{eqnarray}
\begin{array}{ccccccc}
y'_{1}=& a_{11}(t)y_{1}+& a_{12}(t)y_{2}+& a_{13}(t)y_{3}+ &\cdots +& a_{1p}(t)y_{p}+&f_{1}(t)\\
y'_{2}=& \ &a_{22}(t)y_{2}+& a_{23}(t)y_{3}+ &\cdots +& a_{2p}(t)y_{p}+&f_{2}(t)\\
\vdots &  & & & & \vdots& \vdots\\
y'_{p}=& & &\ &\ &a_{pp}(t)y_{p}+&f_{p}(t).\\ 
\end{array}\label{sist_triangular}
\end{eqnarray}
We note that the $p$-th equation in \eqref{sist_triangular} 
can be analyzed by application of Theorem \ref{massera}.
Indeed, by Theorem \ref{massera}-{\it (i)}, we have that
there exist $y_{p}\in AA(\mathbb{R},V)$ given for
\begin{eqnarray*}
y(t)=
\left\{
\begin{array}{lll}
{\displaystyle\int_{-\infty}^{t}} g_p(s,t)f(s)ds,
	&&M(\mathrm{Re}(\mu_p))<0,
	\\
 -{\displaystyle\int_{t}^{\infty}} g_p(s,t)f(s)ds,
	&&M(\mathrm{Re}(\mu_p))>0.
\end{array}
\right.
\quad
\mbox{with}
\quad
 g_p(s,t)=\exp\Big(\int_{s}^{t} a_{pp}(r)dr\Big).
\end{eqnarray*}
Similarly, by substituting $y_{p}\in AA(\mathbb{R},V)$ in $(p-1)-$th
equation of \eqref{sist_triangular}
and by a new application of
Theorem~\ref{massera}-{\it (i)}
we can find an explicit expression
for $y_{p-1}(t)$.
This argument can be repeated to construct
$y_{p-2}(t),y_{p-3}(t),\ldots,y_2(t)$ and $y_1(t)$ 
by backwards substitution and application of
Theorem~\ref{massera}-{\it (i)}
in the system \eqref{sist_triangular}.
Hence, we can construct $y(t)$ an also get that the conclusion of the 
theorem~\ref{massera_matrix}-{\it (i)} is valid.

\vspace{0.5cm}
\noindent
{\it (ii)}
The proof of this item is similar to the proof of the precedent item  
{\it (i)} of the Theorem~\ref{massera_matrix}. In a broad sense,
in this case 
we apply Theorem~\ref{massera}-{\it (ii)} instead
of Theorem~\ref{massera}-{\it (i)} and similarly we use backwards  substitution.
\end{proof}

\begin{thm}
\label{teo:fin_vs_infin}
Let $V$ be a Banach space having the Bohl-Bohr property.
Let $\{\mu_{i}\}_{i=1}^{p}$ be the eigenvalues of 
the $p\times p$ constant matrix $A$ satisfying
$\vert \mu_{i}\vert =1$. 
Then any bounded solution of \eqref{lineal+f} 
$y\in AA(\mathbb{R},V^p)$.
When all the eigenvalues $\mu_i$ are distinct, these solutions have the form
\begin{eqnarray}
y(t)=\exp({At})\left[v+\int_{0}^{t}\exp({-As})f(s)ds \right]
\quad \mbox{for all $v\in V^{p}$}.
\label{sol_sist_lineal_cte}
\end{eqnarray}
In the general case, a formula for the bounded solutions can be also obtained.
\end{thm}
\begin{proof}
If $\{\mu_{i}\}_{i=1}^{p}$ are distinct, the constant system 
\eqref{lineal} is similar to a diagonal system. 
Then, without loss of generality, we can suppose that 
$A$ is an upper triangular matrix. 
Hence, the result \eqref{sol_sist_lineal_cte} 
follows by application of the Theorem \ref{massera_matrix}-{\it (i)}. 
\end{proof}

\begin{thm}
\label{teo:SoL_lineal+f_abstract}
Consider $A:V\rightarrow V$ an infinitesimal generator of
a $C_0$ group of bounded linear operators $T(t)$ with $t \in \mathbb{R}$
and define the Green function
\begin{eqnarray}
G(t,s)=\begin{cases}
T(t-s)P, & t\geq s \\
-T(t-s)(I-P), & t\leq s.
\end{cases}\label{green_c0}
\end{eqnarray}
Assume that $A$ has an $\alpha$-exponential dichotomy, i.e.,
there exist two constants $c$ and $\alpha$ such that
\begin{eqnarray}
\| G(t,s)\| \leq ce^{-\alpha \vert t-s\vert }
\quad
\mbox{for all $(t,s)\in\mathbb{R}^2$.}
\label{green_c0_exp}
\end{eqnarray}
Then if $f\in AA(\mathbb{R},V)$, equation 
\eqref{lineal+f} has a unique solution 
$y\in AA(\mathbb{R},V)$ given by
\begin{eqnarray}
y(t)=\int_{\mathbb{R}} G(t,s)f(s)ds 
\quad
\mbox{for all $t\in\mathbb{R}$}
\label{sol__y}
\end{eqnarray}
and satisfying the following estimate
\begin{eqnarray}
\Vert y\Vert_{\infty}\leq \frac{2c}{\alpha}\Vert f\Vert_{\infty}.
\label{cota_solucion}
\end{eqnarray}
\end{thm}

\begin{proof}
From \eqref{green_c0} and \eqref{green_c0_exp} and $y\in BC(\mathbb{R},V)$
\begin{eqnarray}
\lim_{s\rightarrow \pm \infty} \Vert G(t,s)y(s)\Vert =0 \label{Gy_0}
\end{eqnarray}
Indeed, for  $t\geq s$ we have that $\Vert G(t,s)y(s)\Vert\leq 
ce^{-\alpha(t-s)}\Vert y\Vert_{\infty}$, which implies that
$\Vert G(t,s)y(s)\Vert\to 0$ 
when $s\rightarrow -\infty$. Similarly, 
we get that $\Vert G(t,s)y(s)\Vert$ vanishes when
$s\rightarrow \infty$. 
We note that, for $t\in\mathbb{R}$ (fix), applying $T(t-s)$ on the identity
$y'(s)=Ay(s)+f(s)$
and using the fact that $A$ commutes with $T(t)$ 
on the domain of $A$, we get
\begin{eqnarray}
T(t-s)y'(s) &=& T(t-s)Ay(s)+T(t-s)f(s) \nonumber \\ 
			&=& AT(t-s)y(s)+T(t-s)f(s) \label{3.22}.
\end{eqnarray}
Now, the proof consists of three main parts: (a) we
prove that  the solution $y$ of \eqref{lineal+f} 
is given by \eqref{sol__y}; (b) we prove that $y\in AA(\mathbb{R},V)$
and (c) we prove the uniqueness.

\vspace{0.5cm}
\noindent
{\it (a). Proof of that the solution $y$ of \eqref{lineal+f} 
is given by \eqref{sol__y}}.
Firstly, we note that
a formal integration of \eqref{3.22} on $s\in (-\infty,t)$ gives
the following identity
\begin{eqnarray}
\int_{-\infty}^{t} T(t-s)Py'(s)ds = \int_{-\infty}^{t}AT(t-s)Py(s)ds+\int_{-\infty}^{t}T(t-s)Pf(s)ds. \label{form_integral}
\end{eqnarray}
Moreover we have that
\begin{eqnarray}
\frac{d}{ds} T(t-s)y(s) = -AT(t-s)y(s)+T(t-s)y'(s). \label{derivada_Ty}
\end{eqnarray}
Then an integration on $s\in [r,t]$ implies the following relation
\begin{eqnarray}
 Py(t)-T(t-r)Py(r) = -\int_{r}^{t}AT(t-s)Py(s)ds+\int_{r}^{t}T(t-s)Py'(s)ds.
\end{eqnarray}
Now, by \eqref{Gy_0} letting $r\rightarrow -\infty$, we deduce that
\begin{eqnarray}
Py(t) = -\int_{-\infty}^{t}AT(t-s)Py(s)ds+\int_{-\infty}^{t}T(t-s)Py'(s)ds.\label{Py}
\end{eqnarray}
Here, we note that
a integration of \eqref{3.22} on $s\in [t,\infty)$ with $Q=I-P$ gives
\begin{eqnarray}
\int_{t}^{\infty}T(t-s)Qy'(s) = \int_{t}^{\infty}AT(t-s)Qy(s)ds+\int_{t}^{\infty}T(t-s)Qf(s)ds. \label{form_integral_Q}
\end{eqnarray}
and a integration of \eqref{derivada_Ty} on $s\in [t,r]$ yields
\begin{eqnarray*}
T(t-r)Qy(r)-Qy(t) = \int_{t}^{r}-AT(t-s)Qy(s)ds+\int_{t}^{r}T(t-s)Qy'(s)ds.
\end{eqnarray*}
Now, by \eqref{Gy_0} and letting $r\rightarrow \infty$ in the last relation we get
\begin{eqnarray}
-(I-P)y(t) = -\int_{t}^{\infty}AT(t-s)(I-P)y(s)ds+\int_{t}^{\infty}T(t-s)(I-P)y'(s)ds.
\label{eq:3.20cho}
\end{eqnarray}
The relation \eqref{eq:3.20cho} together 
with \eqref{form_integral_Q} yields
\begin{eqnarray}
-(I-P)y(t)= \int_{t}^{\infty} T(t-s)(I-P)f(s)ds.
\label{-Qy}
\end{eqnarray}
Then, from \eqref{form_integral}, \eqref{Py} and \eqref{-Qy} we obtain
\begin{align}
y(t)&= Py(t)+(I-P)y(t)  \label{3.29}\\
	&= \int_{-\infty}^{t} T(t-s)Pf(s)ds-\int_{t}^{\infty}T(t-s)(I-P)f(s)ds \nonumber \\
	&= \int_{\mathbb{R}} G(t,s)f(s)ds. \nonumber
\end{align}
Thus, we conclude the proof of \eqref{sol__y}.

\vspace{0.5cm}
\noindent
{\it (b). Proof $y\in AA(\mathbb{R},V)$.}
The proof of this property follows by \eqref{3.29}
the hypothesis $f\in AA (\mathbb{R},V)$ and application 
of Lemma~\ref{bi_aa_integrable}.

\vspace{0.5cm}
\noindent
{\it (c).  Proof of uniqueness of bounded solutions.}
The uniqueness of the bounded solution for 
\eqref{lineal+f} follows by the  
fact that $x\equiv 0$ is the unique bounded solution on
$\mathbb{R}$ of the linear equation \eqref{lineal}. 
Indeed, if $x\in BC(\mathbb{R},V)$ is a solution of the 
linear system we have that $x(t)=Px(t)+(I-P)x(t)=x_{1}(t)+x_{2}(t)$. 
Note that $x_{1}\rightarrow \infty$ as $t\rightarrow -\infty$ and 
$x_{2}\rightarrow -\infty$ as $t\rightarrow \infty$, by 
the exponential dichotomy.

\vspace{0.5cm}
Finally, we  note that \eqref{cota_solucion} is a consequence of \eqref{3.29}.
\end{proof}

\begin{thm}\label{cor_unica_solucion}
Assume that $A\in AA(\mathbb{R},\mathbb{C}^{p\times p})$ and  
\eqref{lineal} has an exponential dichotomy 
with a projection $P$ that commutes with the fundamental matrix $\Phi_A(t).$
Assume that $f\in AA(\mathbb{R},V^{p})$. Then, 
the linear non-homogeneous equation \eqref{lineal+f} has
a unique $AA(\mathbb{R},V^{p})$ solution given by
$$y(t)= \int_\mathbb{R} G(t,s)f(s)ds,$$
satisfying \eqref{cota_solucion}.
\end{thm}

\begin{proof}
By application of Lemma~\ref{bi_aa_integrable}.
\end{proof}

\begin{thm}
Consider $V$ be a Hilbert space and $A$ a linear compact 
operator on $V$. Suppose that $V=\oplus_{k=1}^{\infty} V_{k}$ 
is a Hilbert sum such that $V_{k}$ is a finite dimensional 
subspace of $V$ for each $k\in \mathbb{N}$. Suppose that each
orthogonal projection $P_{k}$ on $V_{k}$ commutes with $A$.
If $f\in AA(\mathbb{R},V)$, then every bounded solution $y$ 
of \eqref{lineal+f} is belongs $AA(\mathbb{R},V)$.
\end{thm}

\begin{proof}
Noticing  that for any $y \in V$, we have that
$y=\sum_{k=1}^{\infty}P_ky=\sum_{k=1}^{\infty}y_k$.
Then, by the fact that $A$ is bounded on $V$, we deduce that
\begin{eqnarray}
Ay=\sum_{k=1}^{\infty}Ay_{k}
=\sum_{k=1}^{\infty}AP_{k}y
=\sum_{k=1}^{\infty}P_{k}Ay.
\end{eqnarray}
Now, from the hypothesis that $f\in AA(\mathbb{R},V)$, we have that 
for any subsequence $\{ \tilde{\tau}_{n}\}_{n=1}^{\infty}\subset \mathbb{R}$, 
there exist a subsequence 
$\{ \tau_{n}\}_{n=1}^{\infty}\subset \{ \tilde{\tau}_{n}\}_{n=1}^{\infty}$ 
and a function $\tilde{f}$ such that $f_{\tau_{n}}(t)\to \tilde{f}(t)$
and $\tilde{f}_{-\tau_{n}}(t)\to f(t)$ pointwise on $\mathbb{R}$ when
$n\to \infty$. Then, by compactness of $A$ we deduce that
$Af_{\tau_{n}}(t)\to A\tilde{f}(t)$ and
$A\tilde{f}_{-\tau_{n}}(t)\to Af(t)$ pointwise on $\mathbb{R}$ when
$n\to \infty$.
Now, choosing $y_{k}(t)=P_{k}y(t)$ and assuming that
$y$ is solution of equation \eqref{lineal+f} we can
deduce that
\begin{eqnarray*}
y'_{k}=P_{k}y'  &=& P_{k}(Ay(t)+f(t))\\
				&=& AP_{k}y(t)+P_{k}f(t)\\
				&=& Ay_{k}(t)+P_{k}f(t)
\end{eqnarray*}
or equivalently
$y_{k}$ satisfies the equation \eqref{lineal+f} in the finite 
dimensional space $V_{k}$ with $P_{k}f(t)\in AA(\mathbb{R},V_{k})$
since $P_{k}$ is a bounded linear operator. Thus,
$y_{k}$ is bounded  if and only if 
$y_{k}\in AA(\mathbb{R},V_{k})$. 
Now, if $y(t)$ is bounded the set $\Big\{Ay(t)\vert t\in \mathbb{R}  \Big\}$ 
is relatively compact in $V$. Hence $\sum_{k=1}^{\infty}P_{k}Ay(t)=Ay(t)$ 
uniformly on $\mathbb{R}$.

On the other hand $P_{k}Ay(t)\in AA(\mathbb{R},V_{k})$
since $P_kAy(t)=AP_ky(t)=Ay_{k}(t)$. 
Then, $Ay(t)\in AA(\mathbb{R},V)$ and $y'(t)\in AA(\mathbb{R},V)$
since $y(t)$ satisfies the equation \eqref{lineal+f}. Therefore,  
using the Theorem~\ref{1.1}, $y\in AA(\mathbb{R},V)$ 
since $y\in BC(\mathbb{R},V)$ and $V$ is a Hilbert space.
\end{proof}

\subsection{Results for \eqref{lineal+f+g}}
\label{subsec:lineal+f+g}
Before start we recall the notation $\Delta(\varphi_0,\rho)$ and 
$\varphi_0 $ given on \eqref{eq:condition_L}. Here,
in this subsection, we present two results for \eqref{lineal+f+g}
assuming fundamentally that $g$  satisfies the assumptions given on
 \eqref{eq:condition_L} and  $f$ if a function  such that the inequality
\begin{eqnarray}
\Vert f \Vert \leq \frac{\alpha \rho}{2c},
\label{condition_f}
\end{eqnarray}
is satisfied for some positive constants $\alpha$ and $c$,
then $0\in \Delta(\varphi_0,\rho)$ or equivalently $\Vert \varphi_0 \Vert \leq \rho$.

\begin{thm}\label{theo_sol_unica}
Consider $A(t)\in AA(\mathbb{R},\mathbb{C}^{p\times p})$ such that 
\eqref{lineal} has an $\alpha$-exponential dichotomy with a projection 
$P$ that commutes with $\Phi_A$. Assume that $g$ satisfies the assumptions
given on \eqref{eq:condition_L} and $f$ is selected such that the
\eqref{condition_f} holds. 
Then, if $4cL<\alpha$ the equation \eqref{lineal+f+g}
 has a unique  solution belongs $AA(\mathbb{R},V^p)$.
\end{thm}

\begin{proof}
Let $\Delta=AA(\mathbb{R},V^p)\cap \Delta(\varphi_0,\rho)$ 
and $G=G(t,s)$ the Green matrix associated to $\alpha$-exponential dichotomy, i.e
\begin{eqnarray*}
 \Vert G(t,s)\Vert \leq ce^{-\alpha\vert t-s \vert}
 \quad
 \mbox{for some $c,\alpha\in\mathbb{R}^+$ and for all $(t,s)\in \mathbb{R}^2$.}
\end{eqnarray*}
Now, for $\varphi\in \Delta$,
by Proposition \ref{composition} we follow that the function 
$g(t,\varphi(t))\in AA(\mathbb{R},V^p)$. 
Then, by Theorem~\ref{cor_unica_solucion}, we have that
\begin{eqnarray}
(\Gamma\varphi)(t)=\int_{\mathbb{R}} G(t,s)[f(s)+g(s,\varphi(s))]ds\label{operador}
\end{eqnarray}
is the unique $AA(\mathbb{R},V^p)$ 
solution of the  system \eqref{lineal+f+g}. Moreover, $(\Gamma\varphi)(t)$
satisfies the inequality
\begin{eqnarray*}
\Vert \Gamma\varphi -\varphi_{0} \Vert \leq \frac{2cL}{\alpha} 
\Vert \varphi \Vert \leq \frac{2cL}{\alpha} 
\left( \Vert \varphi -\varphi_0 \Vert +\Vert \varphi_0 \Vert
\right)\leq \frac{4cL}{\alpha}\rho \leq \rho .
\end{eqnarray*}
Thus, $\Delta$ is invariant under $\Gamma$. Furthermore, we note note
that $\Gamma$ is a contraction, since
\begin{eqnarray*}
\vert(\Gamma\varphi_{1})(t)-(\Gamma\varphi_2)(t)\vert 
&\leq& L\int_{\mathbb{R}} \vert G(t,s)\vert \vert  
\varphi_1(s)-\varphi_2(s)\vert ds \\
&\leq& \frac{2cL}{\alpha} \Vert \varphi_1 -\varphi_2 \Vert.
\end{eqnarray*}
Hence by the Banach fixed point arguments, 
we deduce that
$\Gamma$ has a unique fixed point $\varphi \in \Delta$. 
Consequently, 
the function $\varphi$ is the unique solution of \eqref{lineal+f+g}.
\end{proof}

\begin{thm}\label{unique_sol_lineal+f+g}
Consider that $A$ is a linear operator satisfying the
assumptions given on Theorem~\ref{teo:SoL_lineal+f_abstract}.
Assume that $g$ satisfies the assumptions
given on \eqref{eq:condition_L} and $f$ is selected such that the
\eqref{condition_f} holds. 
Then, if $4cL<\alpha$ the equation \eqref{lineal+f+g}
 has a unique mild  solution belongs $AA(\mathbb{R},V^p)$.
\end{thm}

\begin{proof}
Let $\Delta = AA(\mathbb{R},V)\cap \Delta(\varphi_0,\rho)$ and
$\varphi\in AA(\mathbb{R},V)$. Then
$\varphi_{-\tau}\in AA(\mathbb{R},V)$. 
Now, by the composition Proposition~\ref{composition},
we have that $\psi_{\varphi}(\cdot)=g(\cdot,\varphi_{-\tau}(\cdot))\in AA(\mathbb{R},V)$.
Moreover, applying
similar arguments to that used in Theorem~\eqref{theo_sol_unica}, 
we deduce that the Green operator
$$(\Gamma\varphi)(t)=\varphi_0(t)+\int_{\mathbb{R}}G(t,s)\psi_{\varphi}(s)ds$$
maps $\Delta$ into $\Delta$ and additionally we can prove that
$\Gamma :\Delta \rightarrow \Delta$ is a strict contraction. 
Thus, the Banach principle insures the existence of 
a unique $\varphi\in \Delta$ satisfying \eqref{condition_f}.
The proof is now complete.
\end{proof}

\subsection{Results for \eqref{delay_lineal+f+g}}
\label{subsec:delay_lineal+f+g}
In this subsection we generalize the Theorems~\ref{theo_sol_unica}
and \ref{unique_sol_lineal+f+g} to the case of the delay equation
\eqref{delay_lineal+f+g}.

\begin{thm}
\label{teo:delay_lineal+f+g}
Consider that $\tau>0$ is a constant delay. Then, we 
have that the following assertions are valid:
\begin{enumerate} [(i)]
 \item If $A,f$ and $g$ satisfy the hypothesis of Theorem~\ref{theo_sol_unica}. Then,
 the conclusions of Theorem~\ref{theo_sol_unica} holds for 
 the delayed equation \eqref{delay_lineal+f+g}.
 \item If $A,f$ and $g$ satisfy the hypothesis of Theorem~\ref{unique_sol_lineal+f+g}. 
 Then, the conclusions of Theorem~\ref{unique_sol_lineal+f+g} are true for the delayed 
 equation~\ref{delay_lineal+f+g}
\end{enumerate}
\end{thm}

\begin{proof}
The proof follows by application of Theorems \ref{theo_sol_unica} and
\ref{unique_sol_lineal+f+g} since by 
Proposition~\ref{composition} we have that 
$g(\cdot,\phi(\cdot -\tau))\in AA(\mathbb{R},V)$
for every $\phi\in AA(\mathbb{R},\Delta (\varphi_0,\rho))$.
Indeed, this fact is a consequence of the following fact:
 $\phi\in AA(\mathbb{R},V)$ implies  that the 
translation $\phi_{\tau}(\cdot)=\phi(\cdot-\tau)$ is belongs $ AA(\mathbb{R},V)$.
\end{proof}

\section{Application to the Delayed Lasota-Wazewska Model}
\label{sec:appl}

The Lasota-Wazewska model is an autonomous differential equation of the form
\begin{eqnarray}
y'(t)=-\delta y(t)+pe^{-\gamma y(t-\tau )},\ t\geq 0. \label{lasota-wazewska}
\end{eqnarray}
It was occupied by Wazewska-Czyzewska and 
Lasota \cite{lasota-wazewska} to describe the 
survival of red blood cells in the blood of an
animal. In this equation, $y(t)$ describes the
number of red cells bloods in the time $t,\delta >0$ 
is the probability of death of a red blood cell; $p,\gamma$ 
are positive constants related with the production of 
red blood cells by unity of time and $\tau$ is the time
required to produce a red blood cell.

In this section, we study the following delayed model:
\begin{eqnarray}
y'(t)=-\delta(t) y(t)+p(t)g(y(t-\tau)), \label{appli-lasota}
\end{eqnarray}
where $\tau>0$, $\delta(\cdot),\ p(\cdot)$ are positive 
almost automorphic functions and $g(\cdot)$ is a positive 
Lipschitz function with Lipschitz constant $\gamma$.
Equation \eqref{appli-lasota} models several situations 
in the real life, see \cite{liz-martinez}.\\
We will assume the following condition
\begin{enumerate}
\item[(D)] The mean of $\delta$ satisfies $M(\delta)>\delta_{-}>0$. 
\end{enumerate}
In this section, the principal goal is the following Theorem:

\begin{thm}\label{gamma_small_unique_sol}
In the above conditions, for $\gamma$ sufficiently small, 
the equation \eqref{lasota-wazewska} has a unique almost automorphic solution.
\end{thm}
By Lemma \ref{lema_mean}, the linear part of 
equation \eqref{lasota-wazewska} has an exponential dichotomy.
Let $\psi(t)$ be a real almost automorphic 
function and consider the equation
\begin{eqnarray}
y'(t)=-\delta (t)y(t)+p(t)g(\psi(t-\tau)). \label{lasota-psi}
\end{eqnarray}

Then, the bounded solution for the equation \eqref{lasota-psi} satisfies
\begin{eqnarray*}
y(t)=\int_{-\infty}^{t} \exp \left( -\int_{u}^{t}\delta (s)ds  \right)p(u)g(\psi(u-\tau))du.
\end{eqnarray*}

The homogeneous part of equation of \eqref{lasota-psi} 
has an exponential dichotomy and since $\delta$ is almost
automorphic function, by Lemma \ref{bi_aa}, it is integrally
Bi-almost automorphic. Therefore, Theorem \ref{gamma_small_unique_sol} 
follows from Theorem \ref{theo_sol_unica}.

Taking $g(x)=e^{ -\gamma x},\ \alpha>0$, we have the Lasota-Wazewska model:
\begin{eqnarray}
y'(t)=-\delta(t) y(t)+p(t)e^{-\gamma y(t-\tau)},\ t\geq 0. \label{application}
\end{eqnarray}


\begin{cor}
For $\gamma$ small enough, the delayed Lasota-Wazewska model \eqref{application}
has a unique asymptotically stable almost automorphic solution.
\end{cor}

\end{document}